\documentclass[12pt]{amsart}
\usepackage{amssymb}
\usepackage{epsfig}
\usepackage{graphicx}
\numberwithin{equation}{section}

\input xy
\xyoption{all}

\advance\headheight 2pt
\calclayout
\allowdisplaybreaks[3]

\theoremstyle{plain}
\newtheorem{prop}{Proposition}

\newtheorem{thm}[prop]{Theorem}

\newtheorem{lemm}[prop]{Lemma}

\theoremstyle{definition}

\newtheorem{conj}[prop]{Conjecture}

\newtheorem{rema}[prop]{Remark}

\newtheorem{exam}[prop]{Example}

\def\G{{\mathcal G}}   
\def\D{{\mathcal D}}   
\def\I{{\mathcal I}}   

\def\Val{{\mathcal V}}

\def\KK{\boldsymbol{K}}

\def\mo{{\mathfrak o}}
\def\mm{{\mathfrak m}}
\def\ml{{\mathfrak l}}

\def\no{\noindent}
\def\rk{{\rm rk}}
\newcommand{\trdeg}{{\rm tr}\, {\rm deg}}

\newcommand{\Hom}{{\rm Hom}}
\newcommand{\Ker}{{\rm Ker}}

\def\lra{\longrightarrow}
\def\ra{\rightarrow}

\def\A{{\mathbb A}}
\def\C{{\mathbb C}}
\def\P{{\mathbb P}}
\def\Q{{\mathbb Q}}

\def\Z{{\mathbb Z}}
\def\C{{\mathbb C}}
\def\N{{\mathbb N}}

\def\F{{\mathbb F}}

\def\rD{{\rm D}}

\def\rH{{\rm H}}

\def\rK{{\rm K}}
\def\rI{{\rm I}}

\def\rW{{\rm W}}

\def\Syl{{\mathfrak G}}

\begin{document}

\author{Fedor Bogomolov}
\address{Courant Institute of Mathematical Sciences, N.Y.U. \\
 251 Mercer str. \\
 New York, NY 10012, U.S.A.}
\email{bogomolo@cims.nyu.edu}

\author{Yuri Tschinkel}
\address{Courant Institute of Mathematical Sciences, N.Y.U. \\
 251 Mercer str. \\
 New York, NY 10012, U.S.A.}
\email{tschinkel@cims.nyu.edu}

\title{Galois theory and projective geometry}

\begin{abstract}
We explore connections between birational anabelian geometry and 
abstract projective geometry.
One of the applications is a proof of a 
version of the birational section conjecture. 
\end{abstract}

\date{\today}

\maketitle

\section{Introduction}
\label{sect:intro}

A major open problem today is to identify classes of fields {\em characterized} 
by their absolute Galois groups. There exist genuinely different fields with 
isomorphic Galois groups, e.g., $\F_p$ and $\C((t))$.
However, Neukirch and Uchida showed  
that Galois groups of maximal {\em solvable} extensions of number fields or function
fields of curves over finite fields determine the corresponding field, up-to isomorphism 
\cite{neukirch}, \cite{uchida}. 

This result is the first instance of {\em birational anabelian geometry}
which aims to show that Galois groups of certain fields, e.g., function fields of algebraic varieties, 
determine the field, in a functorial way.  
The term {\em anabelian} was proposed by Grothendieck in \cite{groth-letter} 
where he introduced a class of {\em anabelian} varieties, functorially 
characterized by their \'etale fundamental groups; with 
prime examples being hyperbolic curves and varieties successively 
fibered into hyperbolic curves.  
For representative results, see \cite{N}, \cite{V2}, \cite{V91}, \cite{T}, as well as
\cite{ihara-naka}, \cite{Na-Mo}, \cite{pop-alter}, \cite{P-2}, \cite{mochi-topics}. 

However, absolute Galois groups are simply too large. 
It turns out that there are intermediate groups, whose description involves
some {\em projective geometry}, most importantly, 
geometry of lines and points in the projective plane.
These groups are just minimally different from abelian groups;
they encode the geometry of simple configurations. On the other hand, their structure is already
sufficiently rich so that the corresponding objects in the theory 
of fields allow to capture {\em all} invariants and individual properties of large fields, i.e.,
function fields of transcendence degree at least two over algebraically closed ground 
fields. This insight of the first author \cite{B-1}, \cite{B-2}, \cite{B-3}, was
developed in \cite{BT}, \cite{bt0}, and \cite{bt-milnor}. 
One of our main results is that function fields $K=k(X)$ over $k=\bar{\F}_p$ 
are determined by 
$$
\G^c_K:=\G_K/[\G_K,[\G_K,\G_K]],
$$
where $\G_K$ is the maximal pro-$\ell$-quotient of the absolute Galois group $G_K$ of $K$, and $\G^c$ is 
the canonical central extension of its abelianization $\G^a_K$ 
(see also \cite{pop-inv}).

In \cite{bt2} we survey the development of the main ideas merging into this 
{\em almost abelian anabelian geometry} program. 
Here we prove a new result, a version of the 
{\em birational section conjecture} (see Section~\ref{sect:version}). 
In Sections \ref{sect:galois} and \ref{sect:free}
we discuss cohomological properties of Galois groups closely related to the Bloch--Kato conjecture, 
proved by Voevodsky, Rost, and Weibel, and focus on connections to anabelian geometry.

\

\no
{\bf Acknowledgments.} 
The first author was partially supported  by NSF grants DMS-0701578, DMS-1001662, and by the
AG Laboratory GU-HSE grant RF government  ag. 11 11.G34.31.0023.
The second author was partially supported by NSF grants DMS-0739380 and 
0901777.

\section{Projective geometry and $\rK$-theory}
\label{sect:proj}

The introduction of the projective plane 
essentially trivialized plane geometry and provided
simple proofs for many results 
concerning configurations of lines and points, 
considered difficult before that. 
More importantly, the axiomatization efforts in late 19th and early 20th century
revealed that {\em abstract} projective structures capture {\em coordinates}, 
a ``triumph of modern mathematical thought'' \cite[p. v]{whitehead}.  
The axioms can be found in many books, including Emil Artin's lecture notes
from a course he gave at the Courant Institute in the Fall of 1954 \cite[Chapters VI and VII]{artin}. 
The classical result mentioned above is that an abstract projective space
(subject to Pappus' axiom) is a projectivization of a vector space over a {\em field}. 
This can be strengthened as follows:

\begin{thm}
\label{thm:projective}
Let $K/k$ be an extension of fields. Then $K^\times/k^\times$ is simultaneously 
an abelian group and a projective space. Conversely, an abelian group with a compatible projective
structure corresponds to a field extension. 
\end{thm}

\begin{proof}
See \cite[Section 1]{bt2}.
\end{proof}

\


In {\em Algebraic geometry}, projective spaces are the most basic objects. Over nonclosed fields $K$, 
they admit nontrivial forms, called {\em Brauer-Severi} varieties. These forms are classified by the 
{\em Brauer group} ${\rm Br}(K)$, which admits a Galois-cohomological incarnation:
$$
{\rm Br}(K)=\rH^2(G_K, \mathbb G_m).
$$ 
The theory of Brauer groups and the local-global principle for 
{\em Brauer--Severi} varieties over number fields are cornerstones of arithmetic geometry.
Much less is known over more complicated ground fields, e.g., function fields of surfaces.  
Brauer groups, in turn, are closely related to Milnor's $\rK_2$-groups, and more generally $\rK$-theory, 
which emerged in algebra in the study of matrix groups.

We recall the definition of Milnor $\rK$-groups.
Let $K$ be a field. Then 
$$
\rK^{M}_1(K)=K^\times
$$
and the higher $\rK$-groups are spanned by symbols: 
$$
\rK^M_n(K) =  (K^\times)^{\otimes^n} /  \langle \cdots x\otimes (1-x)\cdots \rangle,    
$$
the relations being symbols containing  $x\otimes (1-x)$. 
For $i=1, 2$, Milnor $\rK$-groups of fields coincide with those 
defined by Quillen, and we will often omit the superscript.  

Throughout, we work with function fields of algebraic varieties 
over algebraically closed ground fields; 
by convention, the dimension of the field is its 
transcendence degree over the ground field.

\begin{thm} 
\label{thm:milnor}
\cite{bt-milnor}
Assume that $K$ and $L$ are function fields of algebraic varieties of dimension $\ge 2$, 
over algebraically closed fields $k$ and $l$, and that there exist 
compatible isomorphisms of abelian groups
$$
\rK_1(K)\stackrel{\psi_1}{\lra} \rK_1(L)   \quad \text{ and } \quad \rK_2(K)\stackrel{\psi_2}{\lra} \rK_2(L).
$$
Then there exists an isomorphism of fields 
$$
\psi: K\ra L
$$
such that the induced map on $K^\times$ coincides with $\psi_1^{\pm 1}$. 
\end{thm}

The proof exploits the fact that $\rK_2(K)$ encodes 
the canonical projective structure on $\P_k(K)=K^\times/k^\times$.
It is based on the following observations:
\begin{itemize}
\item The multiplicative groups $k^\times$ and $l^\times$ are characterized as {\em infinitely-divisible} 
elements in $\rK_1(K)$, resp. $\rK_1(L)$. 
This leads to an isomorphism of abelian groups (denoted by the same symbol):
$$
\P_k(K) \stackrel{\psi_1}{\lra} \P_l(L).
$$
\item rational functions $f_1,f_2\in K^\times$ are algebraically dependent in $K$ if and only if their symbol 
$(f_1,f_2)$ is {\em infinitely-divisible}
in $\rK_2(K)$. This allows to characterize $\P_k(E)\subset \P_k(K)$, for one-dimensional $E\subset K$ and we obtain 
a {\em fan} of infinite-dimensional projective subspaces in $\P_k(K)$. 
The compatibility of $\psi_1$ with $\psi_2$ implies that the corresponding structures on $\P_k(K)$ and $\P_l(L)$ coincide. 
\item By Theorem~\ref{thm:projective}, it remains 
to show that $\psi_1$ (or $1/\psi_1$) maps 
projective lines $\P^1\subset \P_k(K)$
to projective lines in $\P_l(L)$. It turns out that projective lines can be intrinsically characterized as
{\em intersections} of well-chosen infinite-dimensional $\P_k(E_1)$ and $\P_k(E_2)$, 
for 1-dimensional subfields $E_1,E_2\subset K$ 
(see \cite[Theorem 22]{bt-milnor} or \cite[Proposition 9]{bt2}).   
\end{itemize}

The theorem proved in \cite{bt-milnor} is stronger, it addresses the case when 
$\psi_1$ is an {\em injective} 
homomorphism.

\section{Projective geometry and Galois groups}
\label{sect:galois-proj}

Let $K$ be a function field over $k=\bar{\F}_p$.
In Section~\ref{sect:proj} we considered the abelian group / projective space $\P_k(K)$ and
its relationship to the $\rK$-theory of the field $K$. Here we focus on a {\em dual} picture. 

Let $R$ be a topological commutative 
ring such that the order of all torsion elements $r\in R$ 
is coprime to $p$. 
Define
\begin{equation}
\label{eqn:ww}
\rW^a_K(R):=\Hom(K^\times/k^\times, R)=\Hom(K^\times, R),
\end{equation}
the $R$-module of continuous homomorphisms, where $\P_k(K)=K^\times/k^\times$ 
is endowed with discrete topology.
We call $\rW^a_K(R)$ the {\em abelian Weil group} of $K$ with values in $R$. 

The abelian Weil group 
carries a collection of distinguished subgroups, corresponding to various {\em valuations}.
Recall that a valuation is a surjective homomorphism 
$$
\nu : K^\times\ra \Gamma_{\nu}
$$
onto an ordered abelian group, subject to a nonarchimedean triangle inequality
\begin{equation}
\label{eqn:tri}
\nu(f+g)\ge \min(\nu(f), \nu(g)). 
\end{equation}
Let $\Val_K$ be the set of all nontrivial valuations of $K$, for $\nu\in \Val_K$ let 
$\mo_{\nu}$ be the valuation ring, $\mm_{\nu}\subset \mo_{\nu}$ its maximal ideal, and $\KK_{\nu}$ the
corresponding function field. We have the following fundamental diagram:

\

\centerline{
\xymatrix{  & 1 \ar[r] & \mo_{\nu}^\times \ar[r]\ar@{=}[d] & K^\times \ar[r]^{\nu} & \Gamma_{\nu}\ar[r] & 1 \\
          1 \ar[r]& (1+\mm_{\nu})^\times \ar[r] &  \mo_{\nu}^\times \ar[r]^{\rho_{\nu}} & \KK_{\nu}^\times\ar[r]  & 1  
}
}

\

Every valuation of $k=\bar{\F}_p$ is trivial, and  
valuation theory over function fields $K=k(X)$ is particularly 
{\em geometric}: all $\nu\in \Val_K$ are trivial on $k^\times$ and define $\Gamma_{\nu}$-valued functions on $\P_k(K)$.
Throughout, we restrict our attention to such ground fields.  
In this context, nontrivial valuations on $k(t)$ are in bijection with points $\mathfrak p\in \P^1$, 
they measure the order of a function at $\mathfrak p$. There are many more valuations on 
higher-dimensional varieties.

We call 
$$
\rI^a_\nu(R):=\{ \gamma\in \rW^a_K(R) \,|\, \gamma \text{ is trivial on } \mo_{\nu}^\times\}\subseteq 
\rW^a_K(R)
$$
the {\em abelian inertia group} and 
$$
\rD^a_\nu(R):=\{ \gamma\in \rW^a_K(R) \,|\, \gamma \text{ is trivial on } (1+\mm_{\nu})^\times\}\subseteq 
\rW^a_K(R)
$$
the {\em abelian decomposition group}, we have
$$
\rD^a_\nu(R)/\rI^a_\nu(R)=\rW^a_{\KK_{\nu}}(R).
$$

\begin{exam}
\label{exam:key}
Let $G_K$ be the absolute Galois group of a field $K$, $G^a_K$ its  abelianization, 
and $\G^a_K$ the $\ell$-completion of $G^a_K$. 
By {\em Kummer theory},
$$
\G^a_K=\rW^a_K(\Z_{\ell}).
$$
Moreover, $\rI^a_\nu(\Z_{\ell})$ and $\rD^a_\nu(\Z_{\ell})$  {\em are} the standard abelian  
inertia and decomposition subgroups corresponding to $\nu$.  
\end{exam}

A valuation $\nu$ defines a simple geometry on 
the projective space $\P_k(K)$; equation \eqref{eqn:tri} implies that 
each finite dimensional subspace $\P^n\subset \P_k(K)$ admits a {\em flag} 
\begin{equation}
\label{eqn:flag}
\P_1\subset \P_2\subset ...
\end{equation}
of projective subspaces, such that 
$$
\nu : \P_k(K)\ra \Gamma_{\nu}
$$ 
is constant on $\P_{j}\setminus \P_{j-1}$, for all $j$, 
and this flag structure is preserved
under multiplicative shifts by any $f\in K^\times/k^\times$.

Let $\P$ be a projective space over $k$, e.g., $\P=\P_k(K)$. 
We say that a map  $\iota : \P\ra R$ to an {\em arbitrary} ring $R$ is a {\em flag map}
if every finite-dimensional $\P^n\subset \P$ admits a flag as in \eqref{eqn:flag} 
such that $\iota$ is constant on each $\P_{j}\setminus \P_{j-1}$; a subset $S\subset \P$ will be called
a {\em flag subset} if its set-theoretic characteristic function is a flag map.

\begin{exam}
\label{exam:flag}
A nonempty flag subset of $\P^1(k)$ is either a point, the complement to a point, or all of $\P^1(k)$. 
Nonempty proper flag subsets of  $\P^2(k)$ are one of the following:
\begin{itemize}
\item a point $\mathfrak p$, a line $\ml$, $\ml^\circ:=\ml\setminus \mathfrak p$,
\item $\P^2\setminus \mathfrak p$, $\P^2\setminus \ml$, $\P^2\setminus \ml^\circ$.
\end{itemize}
\end{exam}

\begin{prop} \cite[Section 2]{BT}
\label{prop:flag-map}
Let $\P$ be a projective space over $k=\bar{\mathbb F}_p$. 
A map $\iota:\P\ra R$ is a flag map if and only if
its restriction to every $\P^1\subset \P$ is a flag map, i.e., is constant on the complement to one point. 
\end{prop}




By \cite[Section 6.3]{BT}, flag maps are closely related to valuations: 
given a flag homomorphism $\iota:\P_k(K)\ra R$
there exists a {\em unique} $\nu\in \Val_K$ and a homomorphism $\chi:\Gamma_{\nu}\ra R$ such that 
\begin{equation}
\label{eqn:iota}
\iota=  \chi  \circ \nu. 
\end{equation}
This means that $\iota\in \rI^a_{\nu}(R)$. 

We now describe the theory of {\em commuting pairs} developed in \cite{BT}: 
We say that nonproportional $\gamma,\gamma'\in \rW^a_K(R)$ form 
a $c$-pair if for every normally closed one-dimensional 
subfield $E\subset K$ the image of the subgroup 
$R\gamma \oplus R \gamma'$ in $\rW^a_E(R)$ is cyclic; 
a $c$-subgroup is a noncyclic 
subgroup $\sigma\subset \rW^a_K(R)$ whose image in $\rW^a_E(R)$ is cyclic, for all $E$ as above.  
We define the {\em centralizer} 
$\mathrm Z_\gamma\subset \rW^a_K(R)$ of an element $\gamma\in \rW^a_K(R)$ 
as the subgroup of all elements forming a $c$-pair with $\gamma$.

The main result of \cite{BT} says:

\begin{thm}
\label{thm:pairs}
Assume that $R$ is one of the following:
$$
\Z, \,\,\hat{\Z},\,\, \Z/\ell^n,\,\, \Z_{\ell}.
$$
Then 
\begin{itemize}
\item 
every $c$-subgroup $\sigma$ has $R$-rank $\le \trdeg_k(K)$;
\item 
for every $c$-subgroup $\sigma$ there exists a valuation 
$\nu\in \Val_K$ such that 
\begin{itemize}
\item 
$\sigma$ is trivial on $(1+\mathfrak m_\nu)^\times\subset K^\times$ 
\item 
there exists a maximal 
subgroup $\sigma'\subseteq \sigma$ of $R$-corank at most one 
such that 
$$
\sigma'\subseteq \Hom(\Gamma_\nu,R)\subset \Hom(K^\times, R)=\rW^a_K(R).
$$
\end{itemize}
\end{itemize}
\end{thm}

The groups $\sigma'$ are, in fact, inertia subgroups $\rI^a_\nu(R)$ 
corresponding to $\nu$. The union of all $\sigma$ containing an inertia subgroup 
$\rI^a_{\nu}(R)$ is  the corresponding decomposition group $\rD^a_{\nu}(R)$. 
If $\rI^a_{\nu}(R)$ is cyclic, generated by $\iota$, 
then $\rD^a_{\nu}(R)=\mathrm Z_{\iota}$, the centralizer of $\iota$.

The proof of Theorem~\ref{thm:pairs} is based on the following geometric observation, 
linking the abelian Weil group with affine/projective geometry: 
Let $\gamma,\gamma'\in \rW^a_K(R)$ be nonproportional elements forming a $c$-pair and let
$$
\begin{array}{ccc}
\P_k(K) &  \stackrel{\phi}{\lra} &  R^2 \\
  f & \mapsto &    (\gamma(f), \gamma'(f))
\end{array}
$$
be the induced map to the affine plane $\mathbb A^2(R)$. 
If follows that 
\begin{itemize}
\item[(*)]
the image of every $\P^1\subset \P_k(K)$ 
satisfies a linear relation, i.e., is contained in an affine line in $\A^2(R)$.
\end{itemize} 
Classically, this is called a {\em collineation}. 
A simple model is a map
$$
\P^2(\F_p)\stackrel{\phi}{\longrightarrow} \A^2(\F_2),
$$
where $p>2$ is a prime and where the target is a set with 4 elements. 
It turns out that when the map $\phi$ satisfies condition $(*)$ then 
the target contains only 3 points. Furthermore, on every line $\P^1\subset \P^2$
the map is constant on the complement of one point! This in turn implies that 
there is a stratification
$$
\mathfrak p \subset \mathfrak l\subset \P^2, 
$$
where $\mathfrak p$ is a point and $\mathfrak l=\P^1$ and $r,r'\in \F_2$, $(r,r')\neq (0,0)$, 
such that 
$\iota: r\gamma+r'\gamma'$ is {\em constant} on $\mathfrak l\setminus \mathfrak p$, and 
$\P^2\setminus \mathfrak l$, i.e., $\iota$ is a flag map on $\P^2$. 
(This last property fails for $p=2$.)
Nevertheless, one can extract the following general fact:
\begin{itemize}
\item if $\gamma,\gamma'$ satisfy $(*)$ then 
there exists a nontrivial $R$-linear combination $\iota:=r\gamma+r'\gamma'$ such that
$$
\iota: \P_k(K)\ra R
$$
is a flag map, i.e., $\iota\in \rI^a_{\nu}(R)$, for some $\nu\in \Val_K$. 
\end{itemize}
The proof is based on a lemma from projective geometry: 

\begin{lemm}
\label{lemm:p2}
Let $\P^2(k)=S_1\sqcup_{j\in J} S_j$ be a 
set-theoretic decomposition into at least three nonempty subsets such that for every 
$\P^1(k)\subset \P^2(k)$ we have 
$$
\P^1(k)\subseteq S_1\sqcup S_j, \,\,\text{ for some }\,\, j\in J,  \quad \text{ or }\,\,
\P^1\subseteq  \sqcup_{j\in J} S_j 
$$ 
Then one of the $S_j$ is a flag subspace of $\P^2(k)$. 
\end{lemm}

These considerations lead to 
the following characterization of multiplicative groups of 
valuation rings, one of the main results of \cite{BT}:

\begin{prop}
\label{prop:nu}
Let $\mo^\times\subset K^\times/k^\times$ be a subgroup such that
its  intersection with any projective plane $\P^2\subset \P_k(K)=K^\times/k^\times$ 
is a flag subspace, i.e., 
its set-theoretic characteristic function is a flag function.
Then there exists a $\nu\in \Val_K$ such that 
$\mo^\times=\mo_{\nu}^\times/k^\times$.
\end{prop}

\

We now turn to more classical objects, namely Galois groups of fields.
Let $G_K$ be the absolute Galois group of a field $K$, $G^a_K$ its abelianization and 
$G^c_K=G_K/[G_K,[G_K,G]]$ its canonical central extension. 
Let  $K$ be a function field of transcendence degree $\ge 2$ over $k=\bar{\F}_p$. 
Fix a prime $\ell\neq p$ and
replace $G^a_K$ and $G^c_K$ by their maximal pro-$\ell$-quotients 
$$
\G^a_K=\Hom(K^\times/k^\times, \Z_{\ell})=\rW^a_K(\Z_{\ell}), \quad \text{ and } \quad \G^c_K.
$$
Note that $\G^a_K$ is a torsion-free $\Z_{\ell}$-module of infinite rank.  

A  {\em commuting pair} is a pair of nonproportional $\gamma,\gamma'\in \G^a_K$ which lift
to commuting elements in $\G^c_K$ (this property does not depend on the choice of the lift).
The main result of the theory of commuting pairs in \cite{BT} says
that if
\begin{itemize}
\item[(**)]
$\gamma, \gamma'\in \G^a_K$ form a commuting pair
\end{itemize}
then the $\Z_{\ell}$-linear span of $\gamma, \gamma'\in \G^a_K$ 
contains an inertia element of some valuation $\nu$ of $K$. 

A key observation is that Property $(**)$ implies
$(*)$, for each $\P^2_k\subset \P_k(K)$, which leads to a flag structure on $\P_k(K)$, 
which in turn gives rise to a valuation. 
For a related result on reconstruction of valuations see 
\cite{efrat-minac-11a}.

Commuting pairs are part of an intricate {\em fan} $\Sigma_K$ on $\G^a_K$, which is defined as the set
of all topologically noncyclic subgroups of $\G^a_K$ which lift to commuting subgroups of $\G^c_K$. 
We have:
\begin{itemize}
\item $\rk_{\Z_{\ell}}(\sigma)\le \trdeg_k(K)$, for all $\sigma\in \Sigma_K$;
\item every $\sigma\in \Sigma_K$ contains a subgroup of corank one which is the inertia
subgroup of some valuation of $K$. 
\end{itemize}

Intersections of subgroups in $\Sigma_K$ reflect subtle dependencies among valuations of $K$. 
Let $\I^a_{\nu}:=\rI^a_{\nu}(\Z_{\ell})\subset \G^a_K$ be the subgroup of 
inertia elements with respect to $\nu\in \Val_K$ and $\D^a_{\nu}=\rD^a_{\nu}(\Z_{\ell})$
the corresponding {\em decomposition group}. In our context, $\D^a_{\nu}$ is also 
the group of all elements $\gamma\in \G^a_K$ forming a commuting pair with every $\iota\in \I^a_{\nu}$. 
As noted above, the definitions of inertia and decomposition groups 
given here are equivalent to the classical definitions in 
Galois theory. We have
$$
\G^a_{\KK_{\nu}} = \D^a_{\nu}/\I^a_{\nu},
$$
the Galois group of the residue field $\KK_{\nu}$ of $\nu$, 
and $\Sigma_{\KK_{\nu}}$ is the set of projections of subgroups of $\sigma\in \Sigma_K$ 
which are contained in $\D^a_{\nu}$. 
When $X$ is a surface, $K=k(X)$, and $\nu$ a {\em divisorial} valuation of $K$, 
we have $\I^a_{\nu}\simeq \Z_{\ell}$ and $\D^a_{\nu}$ is a large group spanned by subgroups $\sigma$ of rank two 
consisting of all element commuting with the topological generator $\delta_{\nu}$ of $\I^a_{\nu}$.

\

To summarize, the Galois group $\G^c_K$ encodes information about
affine and projective structures on $\G^a_K$, in close parallel to what we saw 
in the context of $\rK$-theory in Section~\ref{sect:proj}.
These structures are so individual that they allow to recover the field, via
the reconstruction of the projective structure on $\P_k(K)$:

\begin{thm}
\label{thm:main} \cite{bt0}, \cite{bt1}
Let $K$ and $L$ be function fields of transcendence degree $\ge 2$ over $k=\bar{\F}_p$ and $\ell\neq p$ a prime. 
Let 
$$
\psi^*:\G^a_L\ra \G^a_K
$$
be an isomorphism of abelian pro-$\ell$-groups inducing a bijection of sets
$$
\Sigma_L=\Sigma_K.
$$
Then there exists an isomorphism of perfect closures
$$
\bar{\psi}: \bar{K}\ra \bar{L},
$$ 
unique modulo rescaling by a constant in $\Z_{\ell}^\times$. 
\end{thm}

\section{$\Z$-version of the Galois group}
\label{sect:version}

In this section, we introduce a functorial version of the {\em reconstruction / recognition} 
theories presented in Sections~\ref{sect:proj} and \ref{sect:galois-proj}. 
This  version allows to recover not only field isomorphisms
from data of $\rK$-theoretic or Galois-theoretic type, but also {\em sections}, i.e., 
rational points on varieties over higher-dimensional function fields. 

We work in the following setup: 
let $X$ be an algebraic variety over $k=\bar{\mathbb F}_p$, 
with function field $K=k(X)$. We use the notation:
\begin{itemize}
\item for $x,y\in K^\times$ we let $\ml(x,y)\subset \P_k(K)$ be the projective line
through the images of $x$ and $y$ in $K^\times/k^\times$;
\item for planes $\P^2(1,x,y)\subset \P_k(K)$ 
(the span of lines $\ml(1,x)$, $\ml(1,y)$) we set
$\P^2(1,x,y)^\circ :=\P^2(1,x,y)\setminus \{1\}$.  
\end{itemize}
We will say that $\bar{x},\bar{y}\in K^\times/k^\times$ 
are algebraically dependent, and write $\bar{x}\approx \bar{y}$, 
if this holds for any of their lifts $x,y$ to $K^\times$.

Define the {\em abelian Weil group}
$$
\rW_K^a:=\Hom(K^\times/k^\times, \Z)=\Hom(K^\times,\Z),
$$
with discrete topology. Since $K^\times/k^\times$ is a free abelian group, 
$\rW^a_K$ is a torsion-free, infinite rank $\Z$-module. 
The functor 
$$
K\mapsto \rW^a_K
$$ 
is contravariant; 
it is a $\Z$-version of $\G^a_K$, since it can be regarded as a $\Z$-sublattice of 
$$
\G^a_K=\Hom(K^\times/k^\times,\Z_{\ell}), \quad \ell\neq p.
$$ 

We proceed to explore a functorial version of Theorem~\ref{thm:main}
for $(\rW^a_K,\Sigma_K)$, where $\Sigma_K$ is the corresponding {\em fan}, i.e., 
the set of $c$-subgroups $\sigma\subset \rW^a_K$. 
We work with the following diagram

\

\centerline{
\xymatrix{
 K \ar[r]^{\psi} & L  \\
 K^\times/k^\times \ar[r]^{\psi_1} & L^\times/l^\times \\
 \rW^a_K                          & \ar[l]_{\psi^*} \rW^a_L
}
}

\

\noindent
where $K$ and $L$ are function fields over 
algebraically closed ground fields $k$ and $l$. 
We are interested in situations when 
$\psi_1$
maps subgroups of the form $E^\times/k^\times$, where $E\subset K$ 
is a subfield with $\trdeg_k(E)=1$, into similar subgroups in $L^\times/l^\times$.
The {\em dual} homomorphisms 
\begin{equation}
\label{eqn:psiw}
\psi^*: \rW_L^a\ra \rW_K^a
\end{equation}
to such $\bar{\psi}_1$ 
{\em respect the fans}, in the following sense: 
for all $\sigma\in \Sigma_L$ either 
$\psi^*(\sigma)\in \Sigma_K$ or $\psi^*(\sigma)$ is cyclic. 


\begin{exam}
\label{exam:main}
The following examples of homomorphisms $\psi^*$ 
as in Equation~\ref{eqn:psiw} arise from geometry:
\begin{enumerate}
\item 
If  $X\to Y$ is a dominant map of varieties over $k$ then
$k(Y)= L\subset  K= k(X)$ and the induced homomorphism
$$
\psi^*: \rW_K^a\ra \rW_L^a
$$ 
respects the fans. 
\item Let $\pi : X\ra Y$ be a dominant map of varieties over $k$ and $s:Y\ra X$ a section of $\pi$. 
There exists a valuation $\nu\in \Val_K$ with center $s(Y)$ such that
$$
L^\times \subset \mo_{\nu}^\times \subset K^\times
$$
and the natural projection 
$$
\rho_{\nu}:\mo_{\nu}^\times  \ra \KK_{\nu}^\times
$$ 
onto the multiplicative group of the residue field   
induces an isomorphism
$$
L^\times \simeq \KK_{\nu}^\times
$$
which extends to an isomorphism of fields
$$
L\simeq \KK_{\nu}.
$$
Let $\tilde{\rho}_{\nu}:K^\times \ra L^\times$ be any multiplicative extension of $\rho_{\nu}$ to $K^\times$.
Such $\tilde{\rho}_{\nu}$ map multiplicative subgroups of the form $k(x)^\times$ to similar subgroups of $L^\times$, 
i.e., the dual map
$$
\rW^a_L\ra \rW^a_K
$$
preserves the fans. 
\item More generally, 
let $\nu\in \Val_K$ be a valuation, with valuation ring $\mo_{\nu}$, maximal ideal $\mm_{\nu}$ and residue field $\KK_{\nu}$.
Combining the exact sequences

\

\centerline{
\xymatrix{  & 1 \ar[r] & \mo_{\nu}^\times \ar[r]\ar@{=}[d] & K^\times \ar[r]^{\nu} & \Gamma_{\nu}\ar[r] & 1 \\
          1 \ar[r]& (1+\mm_{\nu})^\times \ar[r] &  \mo_{\nu}^\times \ar[r]^{\rho_{\nu}} & \KK_{\nu}^\times\ar[r]  & 1  
}
}

\

\

\noindent
we have an exact sequence 
$$
1\ra \KK_{\nu}^\times \ra (1+\mm_{\nu})^\times \backslash K^\times \ra \Gamma_{\nu}\ra 1.
$$
Let $\KK_{\nu} = k(Y)$, for some algebraic variety $Y$ over $k$, and let $L$ be any function field containing $\KK_{\nu}=k(Y)$. 
Assume that there is a diagram

\centerline{
\xymatrix{
\mo_{\nu}^\times/k^\times \ar[d]_{\rho_{\nu}} \ar@{^(->}[r] & K^\times/k^\times  \ar[d]^{\psi_1} \\
\KK_{\nu}^\times/k^\times \ar[r]_{\phi} & L^\times/k^\times 
}
}
 
\

\noindent
where $\phi$ is injective and $\psi_1$ 
is an extension of $\phi$. Such extensions exist provided
we are given a splitting of $\nu :K^\times\ra \Gamma_{\nu}$. 
In this case, we will say that $\psi_1$ is {\em defined by a valuation}. 
The dual map to such $\psi_1$ respects the fans. 
Theorem~\ref{thm:main2} shows a converse result: {\em any} $\psi_1$  
respecting one-dimensional subfields can be obtained via this construction.

We proceed to describe the restriction of $\psi_1$ to projective subspaces of $\P_k(K)$, when
$\psi_1$ is obtained from this geometric construction; in particular, $\psi_1$ is not injective. 
In this case, for any $E=k(x)\subset K$ the restriction of
$\bar \psi_1$ to $E^\times/k^\times$ is either
\begin{itemize}
\item 
injective, or
\item 
its restriction to $\ml(1,x)$ is constant on the complement of one point, i.e., 
it factors through a valuation homomorphism.
\end{itemize}
On planes $\P^2(1,x,y)$ 
there are more possibilities: an inertia element 
$$
\iota=\iota_{\nu}\in \rW^a_K=\Hom(K^\times/k^\times,\Z)
$$
restricts to any $\P^2=\P^2_k(1,x,y)$ as a flag map, in particular, it takes at most 
three values. This leads to the following cases:
\begin{enumerate}
\item $\iota$ is constant on $\P^2(1,x,y)$; then
$\P^2(1,x,y)\subset \mo_{\nu}^\times/k^\times$ since
the corresponding linear space does not intersect the maximal ideal
$\mm_{\nu}$ and contains 1. 
The projection 
$$
\rho_\nu : \P^2(1,x,y)\to \KK_{\nu}^\times/k^\times 
$$ 
is injective.
If $x,y$ are algebraically independent the values of $\psi_1$ on 
$\P^2(1,x,y)^\circ$
are algebraically independent in $\KK_{\nu}^\times$.
\item
$\iota$ takes two values and (after an appropriate multiplicative shift)
is constant on  $\P^2(1,x,y)\setminus \{x\}$; in this case
$\psi_1$ on  $\P^2(1,x,y)\setminus \{x\}$ is a composition
of a projection from $x$ onto $\ml(1,y)$ and an embedding of 
$\ml(1,y)\hookrightarrow \KK_{\nu}^\times/k^\times $ with the map $x\to\iota(x)$.
\item
$\iota$ takes two values and is constant on the complement of a projective line, say $\ml(x,y)$; 
then $\psi_1\equiv 1$ on the complement to $\ml(x,y)$.
Note that $1/x \cdot \ml(x,y)\subset \mo_{\nu}^\times$ and hence embeds into
$\KK_{\nu}^\times/k^\times$.
\item
$\iota$ takes three values.
\end{enumerate}
The proof of Theorem~\ref{thm:main2} below relies on a 
reconstruction of these and similar flag structures from fan data, more precisely, 
it involves a construction of a special multiplicative
subset $\mo^\times\subset K^\times/k^\times$ which will be equal to $\mo_{\nu}^\times/k^\times$, 
for some valuation $\nu\in \Val_K$. 
\end{enumerate}
\end{exam}

\begin{thm}
\label{thm:main2}
Assume that $K,L$ are function fields
over algebraic closures of finite fields $k,l$, respectively.
Assume that
\begin{itemize}
\item[(a)] $$
\psi_1: K^\times/k^\times\to  L^\times/l^\times
$$ 
is a homomorphism such that for any one-dimensional
subfield $E\subset K$,
there exists a one-dimensional subfield $F\subset L$ with 
$$
\psi_1(E^\times/k^\times)\subseteq F^\times/l^\times,
$$  
\item[(b)] 
$\psi_1(K^\times/k^\times)$ contains
at least two algebraically independent elements of $L^\times/l^\times$.
\end{itemize}
If $\psi_1$ is not injective then
\begin{enumerate}
\item there is a $\nu\in \Val_K$ such that
$\psi_1$ is trivial on 
$(1+\mm_{\nu})^\times/k^\times \subset \mo_{\nu}^\times/k^\times$; 
\item the restriction of $\psi_1$ to  
\begin{equation}
\label{eqn:restr}
\KK_{\nu}^\times/k^\times = \mo_{\nu}^\times/k^\times(1+\mm_{\nu})^\times \to L^\times/l^\times
\end{equation}
is injective and satisfies (a).
\end{enumerate}

If $\psi_1$ is injective, then there exists a subfield 
$F\subset L$, a field isomorphism 
$$
\phi : K \stackrel{\sim}{\lra} F \subset L,
$$ 
and an integer $m\in \Z$, coprime to $p$, such that 
$\psi_1 $ coincides with the homomorphism induced by
$\phi^m$.
\end{thm}

\

The case of injective $\psi_1$ has been treated in \cite{bt-milnor}. 
The remainder of this section is devoted to the proof of Theorem~\ref{thm:main2} 
in the case when $\psi_1$ is not injective.

\

An immediate application of Condition (a) is: 
if $f\approx f'$, i.e., $f,f'$ are algebraically dependent then $\psi_1(f)\approx \psi_1(f')$. 
The converse does not
hold, in general, and we are lead to introduce the following decomposition:
\begin{equation}
\label{eqn:decom}
\P_k(K) = S_1\sqcup_{f} S_f,
\end{equation}
where 
$$
S_1:=\psi_1^{-1}(1)
$$ 
and, for $f\notin S_1$, 
$$
S_f:=\{ f'\in \P_k(K)\setminus S_1 \,|\,  \psi_1(f')\approx \psi_1(f)\}\subset \P_k(K)
$$ 
are equivalence classes of elements whose images are algebraically dependent.
We record some properties of this decomposition:

\begin{lemm}
\label{lemm:contain}

\

\begin{enumerate}
\item For all $f$, the set $S_1\sqcup S_f$ is closed under multiplication, 
$$
f',f''\in S_1\sqcup S_f \Rightarrow f'\cdot f''\in S_1\sqcup S_f.
$$
\item Every projective line $\ml(1,f)\subset \P_k(K)$, with $\psi_1(f)\neq 1$, is contained in $S_1\sqcup S_f$.
\item Assume that $f,g$ are such that $\psi_1(f),\psi_1(g)$ are 
nonconstant and distinct. If $\psi_1(f)\approx \psi_1(g)$ then $\ml(f,g)\in S_1\sqcup S_f$. Otherwise, 
$\ml(f,g)\cap S_1=\emptyset$.
\item Let $\Pi\subset \P_k(K)$ be a projective subspace such that there exist 
$x,y,z\in \Pi$ with distinct images and such that $\psi_1(x/z)\not\approx \psi_1(y/z)$. 

Then, for any 
$h\in K^\times/k^\times$, the projective subspace $\Pi':=h\cdot \Pi$ satisfies the same property.
\end{enumerate}
\end{lemm}

\begin{proof}
The first property is evident. 
Condition (a) of Theorem~\ref{thm:main2} implies that for every $f\in K^\times/k^\times$
we have
\begin{equation}
\label{eqn:sub}
\psi_1(\ml(1,f))\subset \P_l(F), \quad F=l(\psi_1(f)),
\end{equation}
this implies the second property. 

Considering the shift 
$\ml(f,g)=f\cdot \ml(1,g/f)$ we see that
$$
\psi_1(\ml(f,g)\subseteq \psi_1(f)\cdot \psi_1(\ml(1,g/f)) \subset \psi_1(f)\cdot \P_l(F), 
$$ 
where $F=l(\psi_1(g/f))$. If $\psi_1(f)\approx \psi_1(g/f)$ then $\psi_1(\ml(f,g))\subset S_1\sqcup S_f$. 
Otherwise, $\psi_1(\ml(f,g))$ is disjoint from 1. 

To prove the last property is suffices to remark that $\psi_1(\Pi')$ contains
$\psi_1(hx), \psi_1(hy), \psi_1(hz)$,  and $\psi_1(hx/hz)\not\approx \psi_1(hy/hz)$. 
\end{proof}

\begin{lemm}
\label{lemm:prep}
Let $\Pi=\P^2\subset \P_k(K)$ be a projective plane
satisfying condition (4) of Lemma~\ref{lemm:contain}
and such that the restriction of 
$\psi_1$ to $\Pi$ is not injective.
Then 
\begin{itemize}
\item[(1)] 
there exists a line $\ml\subset \Pi$ 
such that $\psi_1$ is constant on $\Pi\setminus \ml$ or
\item[(2)] 
there exists a $g\in \Pi$ such that 
\begin{itemize}
\item $\psi_1$ is constant on every punctured line $\ml(g,f)\setminus g$;
\item $\psi_1(g)\not\approx \psi_1(f)$, for every $f\neq g$;
\item $\psi_1(f)\approx \psi_1(f')$ for all $f,f'\notin g$.   
\end{itemize}
\end{itemize}  
\end{lemm}

\begin{proof}
After an appropriate shift and relabeling, and using Lemma~\ref{lemm:contain}, 
we may assume that $\Pi=\P^2(1,x,y)$ and that
$\psi_1^{-1}(1)$ contains a nontrivial element $z\in \P^2(1,x,y)$, i.e., $z\in S_1$.
Let 
$$
\P^2(1,x,y)=S_1\sqcup_{f\in \mathcal F} S_f
$$
be the decomposition induced by \eqref{eqn:decom}. 

{\em Step 1.} 
Neither of the sets $S_1$ nor $S_f$ contains a line. 
Indeed, assume that there is a projective line $\ml\subseteq S_1$ and
let $g\in \P^2(1,x,y)\setminus S_1$. 
Every $f'\in \P^2(1,x,y)$  lies on a line through $g$ and $\ml(f',g)$ which intersects $\ml$, and thus $S_1$, 
i.e., all $f'$ lie in $S_1\sqcup S_g$, by Lemma~\ref{lemm:contain}. 
Assume that $\ml\subset S_f$. Every $g\in \P^2(1,x,y)\setminus S_1$ 
lies on a line of the form $\ml(1,g)$, which intersects $\ml\subseteq S_f$. It follows that $g\in S_f$, 
contradicting our assumption that $\psi_1(\P^2(1,x,y))$  contains at least two 
algebraically independent elements.

\

{\em Step 2.}
Split $\mathcal F=\mathcal F'\sqcup \mathcal F''$ into nonempty subsets, arbitrarily, and let 
$$
\P^2(1,x,y)=S_1\sqcup S'\sqcup S'', \quad S':=\sqcup_{f'\in \mathcal F'} S_{f'}, \quad  S':=\sqcup_{f''\in \mathcal F'} S_{f''}
$$
be the induced decomposition.  
By Lemma~\ref{lemm:contain}, every $\ml$ is in either
$$
S_1\sqcup S', \,\, S_1\sqcup S'', \,\, \text{ or } \,\, S'\sqcup S''.
$$
By Lemma~\ref{lemm:p2}, one of these subsets is a flag subset of $\P^2(1,x,y)$.

\

{\em Step 3.}
Assume that $S_1$ is a flag subset. Since it contains 
at least two elements and does not contain
a projective line, by Step 1, we have:
\begin{itemize}
\item 
$S_1=\P^2\setminus \ml$, for some line $\ml$, 
and we are in Case (1), or 
\item 
$S_1=\ml(1,z)^\circ = \ml(1,z)\setminus g$, for some $g\in S'$ (up to relabeling).

Choose a $g''\in S''$, so that  
$\psi_1(g'')\not\approx \psi_1(g)$, and let $\ml$ be a line through $g''$, $\ml\neq \ml(g,g'')$. 
Then $\ml$ intersects $\ml(1,z)^\circ=S_1$, 
which implies that 
the complement $\ml\setminus (\ml\cap \ml(1,z)^\circ) \subseteq S''$. It follows that 
$S''$ contains the complement of $\ml(1,z)\cap \ml(g,g'')$. 
Considering projective lines through $1$ and elements in $S''$ and applying Lemma~\ref{lemm:contain} 
we find that $S'' \supseteq \P^2(1,x,y)\setminus \ml(1,z)$. 
Since $g\notin S''$, we have equality. Thus all elements in $\P^2(1,x,y)\setminus g$ 
have algebraically dependent
images. If $\psi_1$ were not constant on a line $\ml$ through $g$,
with $\ml\neq \ml(1,z)$, let $f_1,f_2\in \ml$
be elements with $\psi_1(f_1)\neq \psi_1(f_2)$.
Lemma~\ref{lemm:contain} implies that $g\in S_{f_1}$, contradicting  our assumption 
that $\psi_1(g)\not\approx\psi_1(f_1)$. 
Thus we are in Case (2).  
\end{itemize}

\

{\em Step 4.}
Assume that $S'$ is a flag subset and $S_1$ is not. We have the following cases:
\begin{itemize}
\item $S'=\{g\}$. Then $S'=S_{g}$ and $\ml(1,g)^\circ:= \ml(1,g)\setminus g\subseteq S_1$.
Assume that there exist  $f_1,f_2\in \P^2(1,x,y)\setminus \ml(1,g)$ with  $\psi_1(f_1)\neq\psi_1(f_2)$. 
If $f_2\notin \ml(g,f_1)$ then $\psi_1(f_1)\approx \psi_1(f_2)$, as $\ml(f_1,f_2)$ intersects $\ml(1,g)^\circ$. 
If there exists at least one $f_2\notin \ml(g,f_1)$ with nonconstant $\psi_1(f_2)$, then 
by the argument above, all elements on the complement to $\ml(1,g)$ are algebraically dependent, 
on every line $\ml$ through $g$, $\psi_1$ is constant on $\ml\setminus g$, and 
we are in Case (2). If $\psi_1$ is identically 1 on $\P^2(1,x,y)\setminus \ml(g,f_1)$ then $\psi_1\neq 1$ on 
$\ml(g,f_1)$ (otherwise $S_1$ would contain a projective line). Thus $S_1$ is a flag subset, contradiction. 
\item 
$S'=\ml^\circ=\ml\setminus g$, for some line $\ml$ and $g\in \ml$. 
Since $S_1$ has at least two elements, 
there is a $z'\in (\P^2\setminus \ml)\cap S_1$.
Lemma~\ref{lemm:contain} implies that $\psi_1$ equals 1 on $\ml(z',g')\setminus g'$, for all 
$g'\in S'=\ml^\circ$. Similarly, $\psi_1$ equals 1 on $\ml(g,z')^\circ:=\ml(g,z')\setminus g$, 
as every point on this punctured line 
lies on a line passing through $S_1$ and intersecting  $S'$. 
Thus $S_1\supseteq \P^2(1,x,y)\setminus \ml$ and  
since $S_1$ does not contain a line, these sets must be equal. 
It follows that we are in Case (1). 
\item
Assume that $S'=\P^2(1,x,y)\setminus \ml$ and we are not in the previous case.
Then $\ml$ contains at least two points in $S_1$ and the complement 
$S''=\ml\setminus (\ml\cap S_1)$ also has at least two points,  $f_1'', f_2''$.  
Thus $S''=S_{f''}$ for some $f''$, by Lemma~\ref{lemm:contain}.  
Since we were choosing the splitting $\mathcal F=\mathcal F'\sqcup \mathcal F''$ arbitrarily, 
we conclude that $S'=S_{f'}$ for some $f'$. 

The same argument as in Step 3. implies that
$\psi_1$ is constant on $\ml(g',f_i'') \setminus f_i''$ for
any $g'\in S'$. Hence  $\psi_1$ is constant  on $\P^2(1,x,y)\setminus \ml$ and we are in Case (1).
\end{itemize}
\end{proof}

\begin{lemm}
\label{lemm:moo}
Let 
$$
\mathfrak u :=\cup\, \ml(1,x) \subseteq K^\times/k^\times=\P_k(K)
$$ 
be union over all lines such that $\psi_1$ is injective on $\ml(1,x)$. 
Assume that there exist nonconstant $x,y\in \mathfrak u$ such that 
$\psi_1(x)\not\approx\psi_1(y)$.
Then 
$$
\mo^\times:=\mathfrak u \cdot \mathfrak u\subset \P_k(K)
$$ 
is a multiplicative subset. 
\end{lemm}

\begin{proof}
First of all, if $x\in \mathfrak u$ then $x^{-1}\in\mathfrak u$, 
since $\ml(1,x)=x\cdot \ml(1,x^{-1})$.

It suffices to show that $\mathfrak u \cdot \mo^\times = \mo^\times$, i.e., 
for all $x,y,z\in \mathfrak u$ one has $xyz=tw$, for some $t,w\in \mathfrak u$.

Assume first that $\psi_1(x)\not\approx\psi_1(y)$ and consider $\P^2(1,x,y^{-1})$. 
A shift of this plane contains the line $\ml(1,xy)$. If $xy\notin \mathfrak u$ then 
$\psi_1$ is not injective on this plane and we may apply Lemma~\ref{lemm:prep}. We are not in Case (1)
and not in Case (2), since $\psi_1$ injects $\ml(1,x)$ and $\ml(1,y^{-1})$, contradiction. 
Thus $\psi_1$ is injective on $\P^2(1,x,y^{-1})$ and 
$xy=t$, for some $t\in \mathfrak u$, which proves the claim. 

Now assume that 
$\psi_1(x),\psi_1(y),\psi_1(z)\in F^\times/l^\times$, for some 1-dimensional $F\subset L$.
By assumption, there exists a $w\in K^\times/k^\times$ such that 
$\psi_1(w)$ is algebraically independent of $F^\times/l^\times$ and such that $\ml(1,w)$ injects.
By the previous argument, $\psi_1$ is injective on the lines $\ml(1,xw)$ and $\ml(1,w^{-1}y)$,
so that $xw,w^{-1}y\in \mathfrak u$. By our assumptions, $\psi_1$ is injective on 
$\ml(1,z)$. Now we repeat the previous argument for 
$xw$ and $z$: there is a $t\in \mathfrak u$ such that
$xw\cdot z = t$.  
\end{proof}

\begin{prop}
\label{prop:general-case}
Let 
$$
\psi_1:K^\times/k^\times\ra L^\times/l^\times
$$ 
be a homomorphism satisfying the assumptions of Theorem~\ref{thm:main2} and
such that $\mathfrak u\subset \P_k(K)$
contains $x,y$ with $\psi_1(x)\not\approx \psi_1(y)$.
Let $\mathfrak o^\times = \mathfrak u\cdot \mathfrak u\subset K^\times/k^\times$ be the
subset defined above. 
Then  
$$
\nu : K^\times/k^\times\to K^\times/\mo^\times
$$
is a valuation homomorphism, i.e., there exists an ordered abelian group $\Gamma_{\nu}$
with $K^\times/\mo^\times\simeq \Gamma_{\nu}$ and $\nu$ is the valuation map. 
\end{prop}

\begin{proof}
By Lemma~\ref{lemm:moo}, we know that $\mo^\times$ is multiplicative. 
We claim that the restriction of $\nu$ to every
$\P^1\subset \P_k(K)$ is a flag map. 
Indeed, we have $\P^1= x\cdot  \ml(1,y)$, for some
$x,y\in K^\times/k^\times$.
On every line $\ml(1,y)\subset \mathfrak u$, the value of $\nu$ is 1, and 
on every line $\ml(1,x)\not\subset \mathfrak u$, the value of $\nu$ is constant on 
the complement to one point. 

By Proposition~\ref{prop:flag-map}, this implies that
$\nu$ is a flag map, i.e., defines a valuation on 
$K^\times/k^\times$ with values on $\Gamma_{\nu}:=K^\times/\mo^\times$. 
\end{proof}

\

To complete the proof of Theorem~\ref{thm:main2} we need to 
treat homomorphisms $\psi_1$ such that for most one-dimensional $E\subset K$, 
the image $\psi_1(E^\times/k^\times)$ is cyclic, or even trivial. We have two cases:
\begin{itemize}
\item[(A)] there exists a line $\ml=(1,x)$ such that $\psi_1$ is not a flag map on this line. 
\item[(B)] $\psi_1$ is a flag map on {\em every line} $\ml\subset\P_k(K)$. 
\end{itemize}

In Case (A), let $\mathfrak u'$ be the union of lines $\ml(1,x)$ such that $\psi_1$ is not 
a flag map on $\ml(1,x)$, i.e., is  not constant on 
the line minus a point. 
By the arguments in Lemma~\ref{lemm:prep}, Lemma~\ref{lemm:moo}, 
and Proposition~\ref{prop:general-case}, 
there exists a unique 1-dimensional normally closed subfield
$F\subset L$ such that $\psi_1(\mathfrak u')\subseteq F^\times/l^\times$.
Define 
$$\mo^\times:=\psi_1^{-1}(F^\times/l^\times)\subset K^\times/k^\times.
$$
By assumption on $\psi_1$, $\mo^\times$ is a proper
multiplicative subset of $K^\times/k^\times$.
The induced homomorphism
$$
\nu: K^\times/k^\times\ra  (K^\times/k^\times)/ \mo^\times =:\Gamma_{\nu}
$$
is a valuation map, since it is a flag
map on every line $\ml\subset \P_k(K)$. Indeed, it is constant on all $\ml(1,x)\subseteq \mathfrak u$
and a valuation map on all $\ml(1,y)\not\subset\mo^\times$.
Hence the same holds for any projective line in $\P_k(K)$ which
implies the result.

In Case (B), we conclude that $\psi_1$ is a flag map on 
$\P_k(K)$ (see, e.g., Lemma 4.16 of \cite{bt0}), i.e., 
there exists a valuation $\nu\in \Val_K$ with value group 
$\Gamma_{\nu}=\psi_1(K^\times/k^\times)$
such that the valuation homomorphism $\nu=\psi_1$.

\begin{rema}
The main steps of the proof  (Lemmas \ref{lemm:contain}, \ref{lemm:prep}, and \ref{lemm:moo}) 
are valid for more general fields: the splitting of $\P^2$ into subsets satisfying
conditions (1) and (2) of Lemma~\ref{lemm:contain} is related 
to a valuation, independently of the ground field (see \cite{BT}).
Here we restricted to $k=\bar{\F}_p$
since in this case the proof avoids some 
technical details which appear in the theory of general valuations.
Furthermore, the condition that $K,L$ are function fields
is also not essential. The only essential property is
the absence of an infinite tower of roots for the elements
of $K,L$ which are not contained in the ground field.
\end{rema}

\section{Galois cohomology}
\label{sect:galois}

By duality, the main result of Section~\ref{sect:version}
confirms the general concept that birational properties of 
algebraic varieties are functorially encoded in the structure of the Galois group $G_K^c$.
On the other hand, it follows from the proof of the Bloch--Kato conjecture that 
$G_K^c$ determines the full cohomology of $G_K$.
Here and in Section~\ref{sect:galois} we discuss
group-theoretic properties of $G_K$ and its Sylow subgroups which we believe are ultimately
responsible for the validity of the Bloch--Kato conjecture.

\

Let $G$ be a profinite group, acting continuously
on a topological $G$-module $M$, and let $\rH^i(G,M)$ be the 
(continuous) $i$-cohomo\-logy group.
These groups are contravariant with respect to $G$ and covariant
with respect to $M$; in most of our applications $M$ either $\Z/\ell$ or $\Q/\Z$,
with trivial $G$-action. We recall some basic properties:
\begin{itemize}
\item $\rH^0(G,M)=M^G$, the submodule of $G$-invariants;
\item $\rH^1(G,M)=\Hom(G,M)$, provided $M$ has trivial $G$-action;
\item $\rH^2(G,M)$ classifies extensions
$$
1\ra M\ra \tilde{G}\ra G\ra 1, 
$$
up to homotopy. 
\item if $G$ is abelian and $M$ finite with trivial $G$-action then 
$$
\rH^n(G,M)=\wedge^n(\rH^1(G,M)), \quad \text{ for all } \quad n\ge 1,
$$
\item if $M=\Z/\ell^m$ then 
$$
\rH^n(G,M)\hookrightarrow \rH^n(\Syl_{\ell},M),  \quad \text{ for all } \quad n\ge 0,
$$
where $\Syl_{\ell}$ is the $\ell$-Sylow subgroup of $G$. 
\end{itemize}
See \cite{milgram} for further background on group cohomology and 
\cite{serre}, \cite{neukirch-schmidt} for background on Galois cohomology.
Let 
$$
G^{(n)}:=[G^{(n-1)},G^{(n-1)}]
$$ the $n$-th term of its {\em derived} series, 
$G^{(1)}=[G,G]$.  
We will write 
$$
G^a=G/[G,G], \quad \text{ and } \quad G^c = G/[[G,G],G]
$$
for the abelianization, respectively, the second {\em lower central series} quotient of $G$. 
Consider the diagram, connecting the first terms in the derived series of $G$
with those in the lower central series:

\

\centerline{
\xymatrix{ 
1 \ar[r] &  G^{(1)}\ar@{>>}[d]\ar[r] & G \ar@{>>}[d]\ar[r] & G^a\ar[r] \ar@{=}[d] & 1 \\
1 \ar[r] &  Z      \ar[r]            & G^c          \ar[r] & G^a  \ar[r]  & 1  
}
}

\

\

We have a homomorphism between $\mathrm E_2$-terms of the spectral sequences computing 
$\rH^n(G,\Z/\ell^m)$ and $\rH^n(G^c,\Z/\ell^m)$, respectively. Suppressing the coefficients, we have

\

\centerline{
\xymatrix{
\rH^p(G^a, \rH^q(Z))\ar@{=}[r] \ar[d]&\mathrm E_2^{p,q}(G^c)\ar[d] \ar@{=>}[r] & \mathrm E^{p+q}(G^c)\ar@{=}[r]\ar[d]&\rH^{p+q}(G^c)\ar[d] \\
\rH^p(G^a, \rH^q(G^{(1)}))\ar@{=}[r] &\mathrm E_2^{p,q}(G)   \ar@{=>}[r] & \mathrm E^{p+q}(G)   \ar@{=}[r] & \rH^{p+q}(G)
}
}

\

We have $\rH^0(G, \Z/\ell^m)=\Z/\ell^m$, for all $m\in \N$, and 
$$
\rH^0(G^a, \rH^1(Z))=\rH^1(Z).
$$  
The diagram of corresponding five term exact sequences takes the form:

\

\centerline{
\xymatrix{ 
\rH^1(G^a) = \rH^1(G)\ar[r] & \rH^1(Z)\ar[d]_{\simeq} \ar[r]^{\hskip 0,2cm d_2}\ar[d]&\rH^2(G^a)\ar[r]\ar@{=}[d] & \rH^2(G) \ar@{=}[d] \\ 
\rH^1(G^a) = \rH^1(G)\ar[r] & \rH^0(G^a, \rH^1(G^{(1)}))\ar[r]^{\hskip 0,5cm d_2'}&\rH^2(G^a)\ar[r]& \rH^2(G) 
} 
}

\

\no
where the left arrows map $\rH^1(G)$ to zero.

Let $K$ be {\em any} field containing $\ell^m$-th roots of 1, for all $m\in \N$, and let $G_K$ be its absolute Galois group. 
We apply the cohomological considerations above to $G_K$. 
By {\em Kummer theory}, 
\begin{equation}
\label{eqn:kummer1}
\rH^1(G_K,\Z/\ell^m)=\rH^1(G_K^a, \Z/\ell^m) = \rK_1(K)/\ell^m 
\end{equation}
and we obtain a diagram

\

\centerline{
\xymatrix{
\rH^1(G^a_K,\Z/\ell^m)\otimes\rH^1(G^a_K,\Z/\ell^m)\ar@{>>}[r]^{\!\!\!\!\!\!\!\!\mathfrak s_K} \ar@{=}[d] &\wedge^2(\rH^1(G^a_K,\Z/\ell^m))= \rH^2(G^a_K, \Z/\ell^m) \\
\rK_1(K)/\ell^m\otimes\rK_1(K)/\ell^m \ar@{>>}[r]^{\hskip 0.7cm\sigma_K}                            &\rK_2(K)/ \ell^m                         
}
}

\

\no
where $\mathfrak s_K$ is the symmetrization homomorphism and $\sigma_K$ is the symbol map, and 
$\Ker(\sigma_K)$ is generated by symbols of the form $f\otimes (1-f)$. 
The Steinberg relations imply that
$$
\Ker(\mathfrak s) \subseteq \Ker(\sigma_K)
$$
(see, e.g., \cite[Section 11]{milnor}). We obtain a diagram

\

\centerline{
\xymatrix{
\rH^1(Z_K,\Z/\ell^m)\,\,\, \ar@{^{(}->}[r]^{d_2} & \rH^2(G^a_K,\Z/\ell^m)   \ar@{=}[d]  \ar[r]^{\pi_a^*} & \rH^2(G_K,\Z/\ell^m) \\ 
         I_K(2)/\ell^m  \,\,\,\ar@{^{(}->}[r]       & \wedge^2(\rH^1(G^a_K, \Z/\ell^m)) \ar@{>>}[r]         &   \rK_2(K)/\ell^m   \ar[u]_{h_K}
}
}

\

\no
where $h_K$ is the {\em Galois symbol} (cf. \cite[Theorem 6.4.2]{neukirch-schmidt}) and $I_K(2)$ is defined by the exact sequence
$$
1\ra I_K(2)\ra \wedge^2(K^\times)\ra \rK_2(K)\ra 1.
$$



A theorem of Merkurjev--Suslin \cite{ms} states that $h_K$ is an isomorphism
\begin{equation}
\label{eqn:ms}
\rH^2(G_K, \Z/\ell^m) = \rK_2(K)/\ell^m. 
\end{equation}
This is equivalent to:
\begin{itemize}
\item $\pi_a^* : \rH^2(G^a_K,\Z/\ell^m)\ra \rH^2(G_K)$ is surjective and
\item $\rH^1(Z_K,\Z/\ell^m) = I_K(2)/\ell^m$.
\end{itemize}

The Bloch--Kato conjecture, proved by Voevodsky, Rost, and Weibel, 
generalizes \eqref{eqn:kummer1} and \eqref{eqn:ms} to all $n$. 
This theorem is of enormous general interest, with far-reaching applications
to algebraic and arithmetic geometry. 
It states that for any field $K$ and any prime $\ell$, one has an isomorphism between 
Galois cohomology and the mod $\ell$ Milnor $\rK$-theory:
\begin{equation}
\label{eqn:kk}
\rH^n(G_K, \mu_\ell^{\otimes n}) = \rK_n^M(K)/\ell.
\end{equation}
It substantially advanced our understanding of relations between 
fields and their Galois groups, in particular, their Galois cohomology. 
Below we will focus on Galois-theoretic consequences of \eqref{eqn:kk}. 

We have canonical central extensions

\

\centerline{
\xymatrix{
        &                               & G_K\ar[d]      \ar[dr]  &                & \\
1\ar[r] &          Z_K\ar[r]\ar@{>>}[d] & G^c_K\ar[r] \ar@{>>}[d] &  G^a_K\ar[r] \ar@{>>}[d] &  1 \\
1\ar[r] & \mathcal Z_K\ar[r]            & \G^c_K\ar[r]^{\pi_a^c}  &  \G^a_K\ar[r]            &  1 
}
}

\

and the diagram

\[
\centerline{
\xymatrix@!{
      & G_K\ar[dl]_{\,\,\,\pi_c} \ar[dr]^{\,\,\,\pi_a} & \\
 \G^c_K \ar[rr]_{\pi^c_a}  &   &           \G^a_K. \\
}
}
\]

\

\noindent
The following theorem relates the Bloch--Kato 
conjecture to statements in Galois-cohomology 
(see also \cite{efrat-1}, \cite{chebolu},
\cite{pos}, and \cite{efrat-minac-11}).

\begin{thm} \cite{B-3}, \cite[Theorem 11]{bt2}
\label{thm:bk}
Let $k=\bar{\F}_p$, $p\neq \ell$, and $K=k(X)$ be the function field of an algebraic variety 
of dimension $\ge 2$. The Bloch--Kato conjecture for $K$ is equivalent to:
\begin{enumerate}
\item 
The map 
$$
\pi_a^* \colon \rH^*(\G^a_K, \Z/\ell^n)\to \rH^*(G_K,\Z/\ell^n)
$$ 
is surjective and
\item 
$\
\Ker(\pi_a^c\circ \pi_c)^* = \Ker(\pi_a^*)$.
\end{enumerate}
\end{thm}

This implies that the Galois cohomology of the pro-$\ell$- quotient $\G_K$ 
of the absolute Galois group $G_K$ encodes important birational information of $X$. 
For example, in the case above, $\G^c_K$, and hence $K$, modulo purely-inseparable
extensions, can be recovered from the cup-products
$$
\rH^1(\G_K,\Z/\ell^n) \times \rH^1(\G_K,\Z/ \ell^n) \ra 
\rH^2(\G_K,\Z/\ell^n), \quad n\in \N.
$$

\section{Freeness}
\label{sect:free}

Let $K$ be a function field over an arbitrary ground field $k$ and $G_K$ the absolute Galois group of $K$.
The pro-$\ell$-quotient $\G_K$ of $K$ is highly individual: for $k=\bar{\F}_p$ it determines $K$ up to purely-inseparable extensions. 
On the other hand, let $\Syl_{\ell}(G_K)$ be an $\ell$-Sylow subgroup of $G_K$. This group is {\em universal}, in the following sense:

\begin{prop} \cite{B-2}
Assume that $X$ has dimension $n$ and that $X$ contains a smooth $k$-rational point. Then 
$$
\Syl_{\ell}(G_K)=\Syl_{\ell}(G_{k(\P^n)}).
$$
\end{prop}

In particular, when $k$ is algebraically closed, 
the $\ell$-Sylow subgroups depend only on the dimension 
of $X$. This universal group will be denoted by $\Syl_{\ell}$.   
The following {\em Freeness} conjecture captures an aspect of this universality. 
It implies the (proved) Bloch--Kato conjecture; but more importantly, 
it provides a structural explanation for its truth. 

\begin{conj}[Bogomolov]
\label{conj:free}
Let $k$ be an algebraically closed field of characteristic $\neq \ell$, 
$X$ an algebraic variety over $k$ of dimension $\ge 2$, 
$K=k(X)$, and write
$$
\Syl_{\ell}^{(1)}:=[\Syl_{\ell},\Syl_{\ell}]
$$
for the commutator of an $\ell$-Sylow subgroup of $G_K$. 
Then 
\begin{equation}
\label{eqn:hi}
\rH^i(\Syl^{(1)}_{\ell}, \Z/\ell^m)=0, \quad\text{ for all }\quad i\ge 2, m\in \N.
\end{equation}
\end{conj}

\begin{rema}
\label{rema:koch}
For profinite $\ell$-groups, the vanishing in Equation~\eqref{eqn:hi} for $i=2$ 
implies the vanishing for {\em all} 
$i\ge 2$ (see \cite{koch}). 
\end{rema}

We now return to the cohomological considerations in Section~\ref{sect:galois}.
The standard spectral sequence associated with 
$$
1\ra G^{(1)} \ra G\ra G^a\ra 1,
$$
gives
$$
\rH^p(G^a, \rH^q(G^{(1)}, \Z/\ell^m))\Rightarrow \rH^n(G,\Z/\ell^m).
$$
We apply this to $G=\Syl_{\ell}$; suppressing the coefficients we obtain:

\

\centerline{
\xymatrix{
               0                          &         0                  &     0                      &\cdots  \\
 \rH^0(G^a, \rH^2(G^{(1)}))               &         0                  &     0                      &\cdots  \\
 \rH^0(G^a, \rH^1(G^{(1)}))\ar[rrd]^{d_2} & \rH^1(G^a, \rH^1(G^{(1)})) &  \rH^2(G^a, \rH^1(G^{(1)}))&\cdots  \\
 \rH^0(G^a, \rH^0(G^{(1)}))               & \rH^1(G^a, \rH^0(G^{(1)})) & \rH^2(G^a, \rH^0(G^{(1)})) &\cdots  \\
}
}

\

\

\noindent
Conjecture~\ref{conj:free} would imply that, for $G=\Syl_{\ell}$, and also for $G=G_K$,  
$\rH^2(G^{(1)})$ and consequently {\em all} 
entries above the second line vanish.
In this case, we have a long exact sequence (see, e.g., \cite[Lemma 2.1.3]{neukirch-schmidt})
$$
0\ra \rH^1(G^a)\ra \rH^1(G)\ra \rH^0(G^a, \rH^1(G^{(1)}))\ra \rH^2(G^a)\ra \cdots 
$$
$$
\cdots\ra
\rH^n(G^a,\rH^1(G^{(1)})) \stackrel{d_2}{\lra} \rH^{n+2}(G^a)\ra \rH^{n+2}(G) \ra \cdots   
$$
In Section~\ref{sect:galois} we saw that for $n=0$ the homomorphism $d_2'$ in the sequence
$$
\rH^n(G^a_K,\rH^1(G^{(1)}_K))\stackrel{d_2'}{\lra} \rH^{n+2}(G^a_K)\ra \rH^{n+2}(G_K) = \rK_{n+2}(K)/\ell^m
$$ 
can be interpreted as the embedding of the {\em skew-symmetric} relations:
$$
\rH^0(G^{a}, \rH^1(Z))= I_K(2)/\ell^m \stackrel{d_2'}{\longrightarrow} \wedge^2(\rH^1(G^a_K))=\wedge^2(K^\times/\ell^m) \ra \rK_2(K)/\ell^m.
$$ 
This relied on Kummer theory and the Merkurjev--Suslin theorem. 
We proceed to interpret the differential $d_2$ for higher $n$.

We work with $G=\Syl_{\ell}$, an $\ell$-Sylow subgroup of the absolute Galois group of a 
function field $K$ over an algebraically closed field. 
We have an exact sequence of continuous $G^a$-modules
$$
1\ra [G,G^{(1)}]\ra G^{(1)}/G^{(2)} \ra Z\ra 1.
$$
Note that $G^a$ acts trivially on $Z$, and $\rH^1(Z)$, and via $x\mapsto gxg^{-1} - x $ on  $G^{(1)}/G^{(2)}$.
Dually we have a sequence of $G^a$-modules:
$$
\Hom([G,G^{(1)}],\Z/\ell^m) \leftarrow \Hom(G^{(1)}/G^{(2)},\Z/\ell^m) \leftarrow \Hom(Z,\Z/\ell^m)\leftarrow 1.
$$
Define 
$$
M:=\Hom(G^{(1)}/G^{(2)},\Z/\ell^m),
$$
then 
$$
M^{G^a}=\Hom(Z,\Z/\ell^m)=\rH^1(Z) = I_K(2)/\ell^m.
$$

We have a homomorphism 
\begin{equation}
\label{eqn:new}
\rH^n(M^{G^a})\ra \rH^n(M)
\end{equation}
via the natural embedding.
Since $\rH^0(M^{G^a})$ embeds into $\rH^2(G^a)$ via $d_2$ we obtain a natural
homomorphism 
$$
t_n :  \rH^n(M^{G^a})\ra \rH^{n+2}(G^a)
$$ 
and the differential $d_2$ on the
image of  $\rH^n(M^{G^a})\to \rH^n(M)$ coincides with $t_n$.
Thus the fact that the
kernel of $\rH^n(G^a)\ra  \rH^n(G)$ is generated by trivial
symbols will follow from the surjectivity of the homomorphism 
in \eqref{eqn:new}.
This, in turn, would follow if the projection $M\to M/M^{G^a}$ defined a trivial
map on cohomology. 
Thus we can formulate the following conjecture which complements
the Freeness conjecture \ref{conj:free}:

\begin{conj} 
The projection $M\to M/M^{G^a}$ can be factored as
$$
M\hookrightarrow D\twoheadrightarrow M/M^{G^a},
$$
where $D$ is a cohomologically trivial $G^a$-module.
\end{conj}

We hope that a construction of a natural module $D$ 
can be achived via algebraic geometry.

\def\cprime{$'$} \def\cftil#1{\ifmmode\setbox7\hbox{$\accent"5E#1$}\else
  \setbox7\hbox{\accent"5E#1}\penalty 10000\relax\fi\raise 1\ht7
  \hbox{\lower1.15ex\hbox to 1\wd7{\hss\accent"7E\hss}}\penalty 10000
  \hskip-1\wd7\penalty 10000\box7}
\providecommand{\bysame}{\leavevmode ---\ }
\providecommand{\og}{``}
\providecommand{\fg}{''}
\providecommand{\smfandname}{\&}
\providecommand{\smfedsname}{\'eds.}
\providecommand{\smfedname}{\'ed.}
\providecommand{\smfmastersthesisname}{M\'emoire}
\providecommand{\smfphdthesisname}{Th\`ese}


\end{document}